\input{epsf}

\newcommand{\Vect}{{\rm Vect}}

\newcommand{\FinVect}{{\rm FinVect}}

\newcommand{\Rep}{{\rm Rep}}

\newcommand{\Ext}{{\rm Ext}}
\newcommand{\Hom}{{\rm Hom}}

\newcommand{\im}{{\rm im}}

\renewcommand{\to}{\rightarrow}

\newcommand{\ten}{\otimes }
\newcommand{\maps}{\colon}

\newcommand{\g}{{\mathfrak g}}

\newcommand{\R}{{\mathbb R}}
\newcommand{\C}{{\mathbb C}}
\newcommand{\F}{{\mathbb F}}

\newcommand{\Z}{\ensuremath{\mathbb{Z}}\xspace}

\newcommand{\Hall}{\mathcal{H}}

\newcommand{\Aut}{{\rm Aut}}
\newcommand{\aut}{{\rm aut}}

\documentclass{article}
\usepackage{amsfonts,amsmath,amssymb,amsthm,amscd,mathrsfs,xspace,multicol}
\usepackage{undertilde}
\usepackage{latexsym}
\usepackage[all]{xy}
\usepackage[breaklinks=true]{hyperref}
\usepackage{graphicx}

\newcommand{\comment}[1]{}
\renewcommand{\u}[1]{\underline{#1}}

\newtheorem{theorem}{Theorem}
\newtheorem{definition}[theorem]{Definition}
\newtheorem{lemma}[theorem]{Lemma}

\newtheorem{proposition}[theorem]{Proposition}

\newtheorem*{theorem*}{Theorem}
\newtheorem*{definition*}{Definition}
\newtheorem*{lemma*}{Lemma}
\newtheorem*{corollary*}{Corollary}
\newtheorem*{proposition*}{Proposition}
\newtheorem*{example*}{Example}
\newtheorem*{conjecture*}{Conjecture}
\newtheorem*{remark*}{Remark}
\newtheorem*{notation*}{Notation}
\newtheorem*{convention*}{Convention}

\newdir{ >}{{}*!/-7pt/@{>}}
\newdir{> >}{*!/0pt/@{>}*!/-5pt/@{>}}

\begin{document}
\sloppy
\title{Hall Algebras as Hopf Objects}
\author{Christopher D.\ Walker\\
       Department of Mathematics, University of California\\
       Riverside, CA 92521 USA}
\maketitle

\begin{abstract}
One problematic feature of Hall algebras is the fact that the standard multiplication and comultiplication maps do not satisfy the bialgebra compatibility condition in the underlying symmetric monoidal category $\Vect$. In the past this problem has been resolved by working with a weaker structure called a `twisted' bialgebra. In this paper we solve the problem differently by first switching to a different underlying category $\Vect^K$ of vector spaces graded by a group $K$ called the Grothendieck group. We equip this category with a nontrivial braiding which depends on the K-grading. With this braiding, we find that the Hall algebra does satisfy the bialgebra condition exactly for the standard multiplication and comultiplication, and can also be equipped with an antipode, making it a Hopf algebra object in $\Vect^K$.
\end{abstract}

\section{Introduction}
Hall algebras have been a popular topic in recent years because of their connection to quantum groups. It is a well known fact, due to Ringel \cite{Ringel}, that the Hall algebra constructed from the representations of a Dynkin quiver over a finite field $\F_q$ is isomorphic to `half' of the quantum group associated to the same Dynkin diagram, namely $U_q^+(\g)$ where $\g$ is the Lie algebra generated by the quiver. This construction provides interesting insight into many structures on the quantum group, but unfortunately does not do everything we hope. 

One of the fundamental problems of Hall algebras arises when we try to make the algebra into a Hopf algebra. In the initial definitions of the Hall algebra, we start with a nice associative multiplication. We also find that the Hall algebra is a coalgebra with an equally nice coassociative comultiplication. However, when we try to check that the algebra and coalgebra fit together to form a bialgebra, we see this fails in the standard underlying category $\Vect$ with its usual braiding. Instead, the combination of these maps obeys ``Green's Formula'', a relationship between the multiplication and comultiplication which we describe in detail below (Proposition \ref{GF}). This formula basically says that the Hall algebra is `almost' a bialgebra in the standard category $\Vect$. Specifically, we only miss being a bialgebra by a coefficient. To see where this extra coefficient comes from, consider the string diagrams which describe the general bialgebra compatibility axiom. As is standard, we will write multiplication of two elements as:
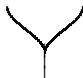
\[\xy
(-5,0);(0,-6)**\crv{(-5,-2)&(0,-4)};
(5,0);(0,-6)**\crv{(5,-2)&(0,-4)};
(0,-6);(0,-9)**\crv{};
\endxy\]
and comultiplication of an element as:
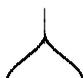
\[\xy
(-5,0);(0,6)**\crv{(-5,2)&(0,4)};
(5,0);(0,6)**\crv{(5,2)&(0,4)};
(0,6);(0,9)**\crv{};
\endxy\]
We can then draw the bialgebra axiom as follows. We first to draw multiplication, follow by comultiplication, which looks like:
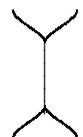
\[ \xy 0;/r.10pc/:
(-10,-20)*{}="b1"; (10,-20)*{}="b2";
(-10,20)*{}="T1"; (10,20)*{}="T2";
(0,10)*{}="C";(0,-10)*{}="D";
"T1";"C"**\crv{(-10,17)&(0,13)};
"T2";"C"**\crv{(10,17)&(0,13)};
"b1";"D"**\crv{(-10,-17)&(0,-13)};
"b2";"D"**\crv{(10,-17)&(0,-13)};
"D";"C"**\crv{};
\endxy \] 
\noindent This should equal the result of comultiplying each element and then multiplying the resulting tensor product of elements. This will look like:
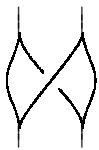
\[\xy 0;/r.10pc/:
(-10,-20)*{}="b1"; (10,-20)*{}="b2";
(-10,20)*{}="t1"; (10,20)*{}="t2";
(-10,15)*{}="A";(10,15)*{}="B";
(-10,-15)*{}="E";(10,-15)*{}="F";
"t1";"A"**\crv{};
"t2";"B"**\crv{};
"A";"E"**\crv{(-10,10)&(-18,0)&(-10,-10)};
"B";"F"**\crv{(10,10)&(18,0)&(10,-10)};
"E";"b1"**\crv{};
"F";"b2"**\crv{};
"B";"E"**\crv{(10,10)&(-10,-10)}\POS?(.5)*{\hole}="H";
"A";"H"**\crv{(-10,10)};
"H";"F"**\crv{(10,-10)};
\endxy \]
But there is wrinkle, namely the braiding of the strings halfway down the diagram. This means we must be working in a braided monoidal category. For the Hall algebra, the seemingly natural choice to work in would be $\Vect$. In $\Vect$ this braiding would simply swap elements with no coefficient. However we have already noted that in $\Vect$ the Hall algebra does not satisfy the bialgebra condition as desired.

To `fix' this, a new structure called a `twisted' bialgebra is usually introduced, where swapping the order of elements can still be done, but at the price of an extra coefficient. This coefficient becomes $q^{-\langle A,D\rangle}$ when swapping elements $A$ and $D$, where $\langle A,D\rangle$ is a bilinear form on a group $K$ (called the Grothendieck group) related to the underlying category of the Hall algebra. 

To obtain a true (untwisted) bialgebra, one then extends the Hall algebra to some larger algebra and alters the multiplication and comultiplication. This process is interesting in its own right, because the result is isomorphic to a larger piece of a quantum group, namely the universal enveloping algebra of the Borel $\mathfrak{b}$. However, we want to take a different direction to avoid the artificial nature of this fix.

In this paper, we will approach the problem directly. Instead of describing the Hall algebra as a `twisted' bialgebra, we will find a braided monoidal category other than $\Vect$ where the Hall algebra is a true bialgebra object. We accomplish this by giving the category of $K$-graded vector spaces, $\Vect^K$, a braiding that encodes the twisting in the Hall algebra. This works since the extra coefficient $q^{-\langle A,D\rangle}$ from Green's Formula depends on the crossing strands in the diagram for the bialgebra axiom. We then accomplishes the same task, but in a way that accounts for the correction factor in the underlying structure, rather than including it later. This idea was mentioned by Kapranov \cite{Kapranov} but details were not provided. Also, Kapranov was working with the same twisted multiplication and comultiplication as Ringel \cite{Ringel2}, where we are using the simpler, non-twisted versions of the maps instead. We will then round out the paper by providing the antipode for this bialgebra to show that the Hall algebra is a Hopf algebra object in our new category.

\section{Hall Algebras}\label{hall}
In this section we will describe the construction of the Ringel-Hall algebra. We begin with a quiver $Q$ (i.e. a directed graph) whose underlying graph is that of a simply-laced Dynkin diagram. We will then consider the abelian category $\Rep(Q)$ of all finite dimensional representation of the quiver $Q$ over a fixed finite field $\mathbb{F}_q$.

We start by fixing a finite field $\F_q$ and a directed graph $D$, 
which might look like this:
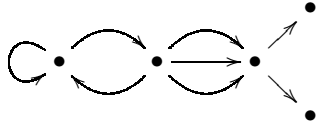
\[ \xymatrix@=10pt{
 &&&&& \bullet  \\
 \bullet \ar@(ul,dl)[] \ar@/^1pc/[rr]
 && \bullet \ar@/^1pc/[ll] \ar@/^1pc/[rr] \ar@/_1pc/[rr] \ar[rr]
 && \bullet \ar[ur] \ar[dr] \\
 &&&&& \bullet  
 }
\]
We shall call the category $Q$ freely generated by $D$ a \textbf{quiver}.
The objects of $Q$ are the vertices of $D$, while the morphisms are
edge paths, with paths of length zero serving as identity morphisms.

By a \textbf{representation} of the quiver $Q$ we mean a functor 
\[   R \maps Q \to \FinVect_q, \]
where $\FinVect_q$ is the category of finite-dimensional vector spaces
over $\F_q$.  Such a representation simply assigns a vector space
$R(d) \in \FinVect_q$ to each vertex of $D$ and a linear operator
$R(e) \maps R(d) \to R(d')$ to each edge $e$ from $d$ to $d'$.  By a
\textbf{morphism} between representations of $Q$ we mean a natural
transformation between such functors.  So, a morphism $\alpha \maps R
\to S$ assigns a linear operator $\alpha_d \maps R(d) \to S(d)$ to
each vertex $d$ of $D$, in such a way that
\[
\xymatrix{
R(d) \ar[d]_{\alpha_d} \ar[r]^{R(e)} & R(d') \ar[d]^{\alpha_{d'}} \\
S(d) \ar[r]_{S(d)} & S(d')
}
\]
commutes for any edge $e$ from $d$ to $d'$.  There is a category
$\Rep(Q)$ where the objects are representations of $Q$ and the
morphisms are as above.  This is an abelian category, so we can speak
of indecomposable objects, short exact sequences, etc.\ in this
category.

In 1972, Gabriel \cite{Gabriel} discovered a remarkable fact.  Namely:
a quiver has finitely many isomorphism classes of indecomposable
representations if and only if its underlying graph, ignoring the
orientation of edges, is a finite disjoint union of Dynkin diagrams of
type $A, D$ or $E$.  These are called {\bf simply laced} Dynkin
diagrams.

Henceforth, for simplicity, we assume the underlying graph of our
quiver $Q$ is a simply laced Dynkin diagram when we ignore the
orientations of its edges.  Let $X$ be the underlying groupoid of
$\Rep(Q)$: that is, the groupoid with representations of $Q$ as
objects and \textit{isomorphisms} between these as morphisms.  We will
use this groupoid to construct the Hall algebra of $Q$.

As a vector space, the Hall algebra is just $\R[\u{X}]$.
Recall that this is the vector space
whose basis consists of isomorphism classes of objects in $X$.  In
fancier language, it is the zeroth homology of $X$.

We now focus our attention on the Hall algebra product.  Given three quiver
representations $M,N,$ and $E$, we define the set:
\[\mathcal{P}_{MN}^E=
\{(f,g): 
0\to N \stackrel{f}{\rightarrow} E \stackrel{g}{\rightarrow} M \to 0 
\textrm{\; is exact} \} \]
and we call its set cardinality $P_{MN}^E$. In the chosen category this set has a finite cardinality, since each representation is a finite-dimensional vector space over a finite field. 
The Hall algebra product counts these exact sequences, but with a
subtle `correction factor':
\[[M] \cdot [N] =\sum_{[E] \in \u{X}} 
\frac{P_{MN}^E}{\aut(M) \, \aut(N)}\, [E] \,.\]
Where we call $\aut(M)$ the set cardinality of the group $\Aut(M)$.

Somewhat surprisingly, the above product is associative.  In fact,
Ringel \cite{Ringel} showed that the resulting algebra is isomorphic
to the positive part $U_q^+ \g$ of the quantum group corresponding to
our simply laced Dynkin diagram!  So, roughly speaking, the Hall algebra
of a simply laced quiver is `half of a quantum group'.

This isomorphism also relates to a coalgebra structure on the Hall algebra. Using the same ideas from the multiplication formula, we can define a comultiplication on the Hall algebra to be a carefully weighted sum on ways to `factor' a representation via short exact sequences. Formulaically this becomes:
\[\Delta(E)=\sum_{[M],[N] \in \u{X}} 
\frac{|\mathcal{P}_{MN}^E|}{ \aut(E)}\, [N]\ten [M] \,.\]
Again, Ringel found that these are the correct factor to make the comultiplication coassociative. However, we immediately run into a problem; these two maps do not satisfy the compatibility condition for a bialgebra.

\section{The Category of K-graded Vector Spaces}\label{gvs}

It is interesting to note that the standard multiplication and comultiplication on $U_q^+ \g$ (which the Hall algebra is isomorphic to) also do not satisfy the compatibility axiom of a bialgebra, so we should not expect the Hall algebra to, either. This does not mean there is not an interesting relationship between the multiplication and comultiplication in the Hall algebra. This relationship is often described as being a `twisted' bialgebra, where we do not use the standard extension of the multiplication to the tensor product. We would like to take a different point of view here. It turns out that the bialgebra axiom can be satisfied if we change the category in which we ask for them to be compatible. 

In order to describe this new category, we will start with a definition of the Grothendieck group of a general abelian category.
\begin{definition}
Let $\mathcal{A}$ be an abelian category. We can define an equivalence relation on isomorphism classes of objects in $\mathcal{A}$ by $[A]+[B]=[C]$ if there exists a short exact sequence $0\to A\to C\to B\to 0$. The set of equivalence classes under this relation form a group $K_0(\mathcal{A})$ called the {\rm \textbf{Grothendieck group}}.
\end{definition} 
$K_0(\mathcal{A})$ has a universal property in the following sense. Given any abelian group $G$, any additive function $f$ from isomorphism classes of $\mathcal{A}$ to the group $G$ will give a unique abelian group homomorphism $\tilde{f}\maps K_0(\mathcal{A})\to G$ such that the following diagram commutes:

\[\xymatrix{
\mathcal{A}\ar[rr]\ar[dr]_f &  &  K_0(\mathcal{A}) \ar[dl]^{\exists !\tilde{f}} \\
  & G & \\
}\]

The original purpose of the Grothendieck group was to study Euler characteristics, and this is precisely why we are interested in them here.

In many of the standard references for Hall algebras \cite{Hubery, Schiffman} the characteristics of the Grothendieck group of $\Rep(Q)$ are explained explicitly. Many of these properties follow from the fact that $\Rep(Q)$ is hereditary.  We can also describe these properties in the general case of an abelian category $\mathcal{A}$ which has finite homological dimension. However, to construct the entire Hall algebra, our
abelian category will need to hold to the extra finiteness properties that the
groups $\Ext^i(M,N)$ must be finite. This condition is sufficient since it makes the sets $\mathcal{P}^E_{MN}$ finite, and makes the bilinear form in the next proposition well defined.

\begin{proposition}Let $\mathcal{A}$ be an abelian $k$-linear category for some field $k$. Suppose that $\mathcal{A}$ has finite homological dimension $d$ and $\dim\Ext^i(M,N)$ is finite for all objects $M,N\in \mathcal{A}$. If $K=K_0(\mathcal{A})$ is the Grothendieck group of $\mathcal{A}$, then $K$ admits a bilinear form $\langle\cdot,\cdot\rangle\maps K\times K\to \C$ given by:
\[\langle \underline{m},\underline{n}\rangle = \sum_{i=0}^d (-1)^{i}\dim\Ext^i(M,N)\]
\end{proposition}

\begin{proof}
We prove the theorem for $d=1$ (i.e. when the category is hereditary) since this is the main case we will use. The case when $d=0$ is simply bilinearity of $\Hom$, and the cases where $d>1$ follows by a similar argument to $d=1$.\\

For $d=1$ we need to show that:
\[\dim\Hom(M,N_1\oplus N_2)-\dim\Ext^1(M,N_1\oplus N_2)=\]
\[\dim\Hom(M,N_1)-\dim\Ext^1(M,N_1)+\dim\Hom(M,N_2)-\dim\Ext^1(M,N_2)\]
we begin with the short exact sequence:
\[0\to N_1\stackrel{i_1}{\rightarrow} N_1\oplus N_2 \stackrel{\pi_2}{\rightarrow} N_2\to 0\]
which, since $d=1$, gives rise to the long exact sequence:
\[0\to \Hom(M,N_1) \stackrel{}{\rightarrow} \Hom(M,N_1\oplus N_2) \stackrel{}{\rightarrow} \Hom(M,N_2) \stackrel{h}{\rightarrow}\] \[\Ext^1(M,N_1)\stackrel{}{\rightarrow} \Ext^1(M,N_1\oplus N_2) \stackrel{}{\rightarrow} \Ext^1(M,N_2) \to 0.\]
Using a variety of basic equations from the fact that this sequence is exact, as well as some dimension arguments, The left hand sides becomes:
\[\dim\Hom(M,N_1\oplus N_2)-\dim\Ext^1(M,N_1\oplus N_2)\]
\[=\dim \im \tilde{\pi_2}+\dim \ker \tilde{\pi_2}-\dim \im \hat{\pi_2}-\dim \ker \hat{\pi_2}\]
and the right hand side turns into:
\[\dim\Hom(M,N_1)-\dim\Ext^1(M,N_1)+\dim\Hom(M,N_2)-\dim\Ext^1(M,N_2)\]
\[=\dim \im h+\dim \ker h+\dim\im \tilde{i_1}-\dim \im \hat{i_1}-\dim \ker \hat{i_1}-\dim\im \hat{\pi_2}\]
\[=\dim \ker \hat{i_1}+\dim \im \tilde{\pi_2}+\dim\ker \tilde{\pi_2}-\dim \ker \hat{\pi_2}-\dim \ker \hat{i_1}-\dim\im \hat{\pi_2}\]
\[=\dim \im \tilde{\pi_2}+\dim\ker \tilde{\pi_2}-\dim \ker \hat{\pi_2}-\dim\im \hat{\pi_2}.\]
\end{proof}

In general, it is possible to construct a braided monoidal category $\Vect^G$ from any abelian group $G$ equipped with a bilinear form $\langle\cdot,\cdot\rangle$. One common example is the category of super-algebras, which can be thought of in this context in terms of the group $\Z_2$ with its unique non-trivial bilinear form. Joyal and Street \cite{JoyalStreet} described the general idea of constructing a braided monoidal category from a bilinear form. In the next theorem, we will describe how this braiding works in detail for our desired case of the Grothendieck group $K=K_0(\mathcal{A})$ with the previously described bilinear form. 

\begin{theorem}\label{BMC} Let $\mathcal{A}$ be an abelian, $k$-linear category with finite homological dimension. Let $K=K_0(\mathcal{A})$ be its Grothendieck group, and suppose $\dim\Ext^i(M,N)$ is finite for all objects $M,N \in\mathcal{A}$. Then, the category $\Vect^K$ of $K$-graded vector spaces and grade preserving linear operators is a braided monoidal category, with the braiding given by:
\[B_{V,W}:V\ten W\to W\ten V\]
\[v\ten w \mapsto q^{-\langle \underline{n},\underline{m}\rangle}w\ten v\]
where $q$ is a non-zero element of $k$.
\end{theorem}

\begin{proof} The monoidal structure on this category is just the tensor product in the category $\Vect$. To define a braiding on this category, we first note that the braiding is defined by isomorphisms in the category which are graded linear operators. Because of linearity, it is enough to define these isomorphisms on a single graded piece. Also note that for any two $K$-graded vector spaces $V$ and $W$, a graded piece of the tensor product $V\ten W$ can be written as the sum of tensor products of graded pieces from $V$ and $W$, or more precisely:
\[(V\ten W)_{\underline{d}}=\bigoplus_{\underline{n}\in K} V_{\underline{n}}\ten W_{\underline{d}-\underline{n}}.\]
This lets us define the braiding $B_{V,W}\maps V\ten W\to W\ten V$ only on the tensor product of the graded piece $V_{\underline{n}}\ten W_{\underline{m}}$. 
We thus define the map: 
\[B_{\underline{n},\underline{m}}\maps V_{\underline{n}}\ten W_{\underline{m}}\to W_{\underline{m}}\ten V_{\underline{n}}\]
\[v\ten w \mapsto q^{-\langle \underline{n},\underline{m}\rangle}w\ten v\]
which is easily seen to be an isomorphism. We only need to check the hexagon equations, i.e. ones of the form: 
\[\xymatrix{
             & (W\ten V)\ten U\ar[r]^\alpha & W\ten (V\ten U)\ar[dr]^{1\ten B_{V,U}} & \\
(V\ten W)\ten U \ar[ur]^{B_{V,W}\ten 1}\ar[dr]_\alpha & & & W\ten (U\ten V) \\
             & V\ten (W\ten U) \ar[r]_{B_{V,W\ten U}} & (W\ten U)\ten V\ar[ur]_\alpha & \\
}\]
We will make the argument for the above hexagon identity, noting the the other versions follow by a similar argument. Now, since we have restricted ourselves to vector spaces with a single grade, it is enough to chase a general element around this diagram. let $v\in V_{\underline{n}}$, $w\in W_{\underline{m}}$, and $u\in U_{\underline{p}}$. The top path of the hexagon diagram yields the composite:
\[(v\ten w)\ten u \mapsto q^{-\langle\underline{n},\underline{m}\rangle-\langle\underline{n},\underline{p}\rangle}w\ten(u\ten v).\]
For the bottom path we note that $v\ten w\in (V\ten W)_{\underline{m+p}}$, so we get the composite:
\[(v \ten w)\ten u \mapsto q^{-\langle \underline{n}, \underline{m+p}\rangle} w \ten(u \ten v).\]
Hence, commutativity of the diagram will follow from the equality \[-\langle\underline{m},\underline{n}\rangle-\langle\underline{m},\underline{p}\rangle=-\langle \underline{m}, \underline{n+p}\rangle,\] which is precisely bilinearity of the form $\langle\cdot, \cdot\rangle$.
\end{proof}

\section{The Hopf Algebra Structure}
Now we consider our Hall algebra in the braided monoidal category $\Vect^K$. The concept of a Hopf algebra object in a braided monoidal category was described by Majid \cite{Majid}, where he called it a `braided group', but later \cite{Majid2} described it in the way we will use here. The basic idea is to ask if the standard defining commutative diagrams for a Hopf algebra hold in some braided monoidal category, instead of the symmetric monoidal category $\Vect$. For the remainder of this section, we will let $Q$ be a simply laced Dynkin quiver. We will focus back on the specific abelian category $\Rep(Q)$ and the category of $K_0(\Rep(Q))-$graded vector spaces, which we showed in Section \ref{gvs} to be a braided monoidal category. Remember that $\Rep(Q)$ is hereditary and satisfies all the finiteness conditions of Section \ref{gvs}. We can now state the main theorem of this paper.
\begin{theorem}\label{HO} The Hall algebra of $\Rep(Q)$ is a Hopf algebra object in the category $Vect^K$. \end{theorem}
To prove this theorem we need to work through the following lemmas. For what follows, we will set $X$ to be the underlying groupoid of $\Rep(Q)$, $\underline{X}$ to be the set of isomorphism classes in $X$, and $K=K_0(\Rep(Q))$. Recall that $R[\underline{X}]$ is the vector space of all finite linear combinations of elements of $\u{X}$. This vector space, which is the underlying vector space of the Hall algebra, is easily seen to be $K$ graded.
\begin{lemma}
The vector space $\Hall=\R[\underline{X}]$ is a $K$-graded vector space, with the grading on each isomorphism class $[M]\in\u{X}$ given by its image in $K$.
\end{lemma}

For the next two lemmas, we note that the multiplication and comultiplication described were shown to be associative and coassociative in the original category $\Vect$ by Ringel \cite{Ringel}. This fact passes to our new category since neither axiom requires or depends on the particular braiding on vector spaces, so we will not repeat the argument. After stating both lemmas, we will provide a brief description of why each one is a morphism in the new category $\Vect^K$.

\begin{lemma}\label{mult} The multiplication map $m:\Hall\ten \Hall\to \Hall$ defined on basis elements by:

\[m([M]\ten[N]) =\sum_{[E]} 
\frac{P_{MN}^E}{\aut(M) \, \aut(N)}\, [E]\]
is a morphism in $\Vect^K$.
\end{lemma}

\begin{lemma}\label{comult} The comultiplication map $\Delta:\Hall\to \Hall\ten \Hall$ defined on basis elements by:

\[\Delta(E)=\sum_{[M],[N]} 
\frac{P_{MN}^E}{ \aut(E)}\, [N]\ten [M]\]
is a morphism in $\Vect^K$. 
\end{lemma}
Note that when Q is a simply-laced Dynkin quiver, the sums in Lemmas \ref{mult} and \ref{comult} are finite. Both of these lemmas are true for a similar reason. The important fact to note here is that for a fixed $M$, $N$, and $E$ in either sum, there is a short exact sequence $0\to N\to E\to M\to 0$. So by the definition of the Grothendieck group $K$, we have that their images obey the identity $[M]+[N] = [E]$. These images determine the grade of the corresponding graded piece they sit in, so the grade is clearly preserved by both maps.

Now we can focus on the compatibility of the new maps, which was the main reason for constructing this new category. We first need an important identity for the multiplication and comultiplication known as Green's Formula.

\begin{proposition}\label{GF}{\rm (Green's Formula).} For all $M$, $N$, $X$, and $Y$ in $\Rep(Q)$ we have the identity:
\[\sum_{[E]}{\frac{P^E_{MN}P^E_{XY}}{\aut(E)}}=\sum_{[A],[B],[C],[D]}{q^{-\langle A,D\rangle}\frac{P^{M}_{AB}P^{N}_{CD}P^{X}_{AC}P^{Y}_{BD}}{\aut(A)\aut(B)\aut(C)\aut(D)}}.\]
\end{proposition}
The proof of Green's formula is quite complex, and involves a large amount of homological algebra. It was first presented by Ringel \cite{RingelGreen}, and also appears in \cite{Hubery} and \cite{Schiffman} with good explanations. What we are interested in is the consequence of Green's formula.

We observe in Green's formula the presence of our braiding coefficient $q^{-\langle A,D \rangle}$. It is important to note that this coefficient depends on what some might view as the ``outside'' objects $A$ and $D$, and not the ``inside'' objects $B$ and $C$. We deal with this by using a different comultiplication than the one usually described in the literature \cite{Hubery, Schiffman}. In fact, in the category $\Vect$ our chosen comultiplication is the opposite of the standard choice.

\begin{lemma}
In the category $\Vect^K$ the multiplication $m$ and comultiplication $\Delta$ satisfy the bialgebra condition, and thus $\Hall$ is a bialgebra object in $\Vect^K$.
\end{lemma}

\begin{proof}
All the hard work for this proof was done in proving Green's Formula. We now just need to check that Green's Formula gives us the bialgebra compatibility. First we will multiply two objects, then comultiply the result to get:
\[\begin{array}{rl}
\Delta([M]\cdot[N]) & = \displaystyle{\sum_{[E]} \frac{P^E_{MN}}{\aut(M)\aut(N)}\Delta([E])} \\
                    & = \displaystyle{\sum_{[X],[Y]}\sum_{[E]} \frac{P^E_{MN}P^E_{XY}}{\aut(M)\aut(N)\aut(E)}[Y]\ten [X]}\\
\end{array}
\]
On the other hand, if we first comultiply each object, then multiply the resulting tensor products we have:
\[\Delta([M])\cdot \Delta([N]) = \sum_{[A],[B],[C],[D]} \frac{P^M_{AB}P^N_{CD}}{\aut(M)\aut(N)} ([B]\ten[A])\cdot ([D]\ten [C])\]
To continue, we need to remember the in our category $\Vect^K$ the braiding is non-trivial. This means that if we want to extend the multiplication on $\mathcal{H}$ to $\mathcal{H}\ten \mathcal{H}$ we must include the braiding coefficient. Specifically, we get the formula:
\[([B]\ten [A])\cdot ([D]\ten [C]) = q^{-\langle A,D\rangle} [B]\cdot[D]\ten [A]\cdot[C]\]
When substituted above, this yields:
\[\begin{array}{c}
\displaystyle{\sum_{[A],[B],[C],[D]} \frac{P^M_{AB}P^N_{CD}}{\aut(M)\aut(N)} ([B]\ten[A])\cdot ([D]\ten [C])}\\
 = \displaystyle{\sum_{[A],[B],[C],[D]} q^{-\langle A,D\rangle}\frac{P^M_{AB}P^N_{CD}}{\aut(M)\aut(N)} [B]\cdot[D]\ten [A]\cdot[C]}\\
 = \displaystyle{\sum_{[X],[Y]}\sum_{[A],[B],[C],[D]}\frac{q^{-\langle A,D\rangle} P^{M}_{AB}P^{N}_{CD}P^{X}_{AC}P^{Y}_{BD}}{\aut(M)\aut(N)\aut(A)\aut(B)\aut(C)\aut(D)}[Y]\ten[X]}\\
\end{array}
\]
Thus, Green's formula give the equality of the two sides.
\end{proof}

For completeness, we will also define an antipode for this bialgebra object to make it a Hopf object. This map is also a morphism in $\Vect^K$ since it clearly preserves the grading.
\begin{lemma}
The map $S:\Hall\to \Hall$ defined on generators by:
\[S([M])=-[M]\]
is a $K$-grade preserving linear operator, and is an antipode for $\Hall$. Thus $\Hall$ is a Hopf algebra object in $\Vect^K$.
\end{lemma}

It is possible to generalize these results to other abelian categories, provided they obey the same finiteness properties as $\Rep(Q)$.

\begin{theorem}
Let $\mathcal{A}$ be an abelian, $k$-linear, hereditary category, where $k=\mathbb{F}_q$. Let $K=K_0(\mathcal{A})$ be its Grothendieck group, and suppose $\dim\Ext^i(M,N)$ is finite for all objects $M,N \in\mathcal{A}$. If the sum 
\[\sum_{[M],[N]} 
\frac{P_{MN}^E}{ \aut(E)}\, [N]\ten [M]\]
is finite for all objects $E\in \mathcal{A}$, then the Hall algebra $\mathcal{H}(\mathcal{A})$ is a Hopf object in $\Vect^K$.
\end{theorem}

\begin{proof}
Examining the proof of Theorems \ref{BMC} and \ref{HO}, we see these are the conditions that we need to generalize the result from the case $\mathcal{A} = \Rep(Q)$ to other abelian categories. Specifically, we need hereditary to prove Green's Theorem.
\end{proof}

\end{document}